\documentclass[11pt]{amsart}
\usepackage{fullpage}   

\usepackage{amsmath}
\usepackage{amsfonts}
\usepackage{amssymb}
\usepackage{amsthm}
\usepackage{subfigure} 
\usepackage{graphicx}
\usepackage{color}
\usepackage{enumerate}
\usepackage{pinlabel}
\usepackage{comment}

\newtheorem{theorem}{Theorem}[section]
\newtheorem{proposition}[theorem]{Proposition}
\newtheorem{lemma}[theorem]{Lemma}

\newtheorem*{maintheorem}{Main Theorem}

\theoremstyle{definition}

\newtheorem{conjecture}[theorem]{Conjecture}
\newtheorem{example}[theorem]{Example}

\theoremstyle{remark}

\numberwithin{equation}{section}

\begin{document}


\title{Excluded minors and the ribbon graphs of knots}

\author[I.~Moffatt]{Iain Moffatt}
\address{Department of Mathematics,
Royal Holloway,
University of London,
Egham,
Surrey,
TW20 0EX,
United Kingdom.}
\email{iain.moffatt@rhul.ac.uk}

\date{\today}

\begin{abstract}  
In this paper we consider minors of ribbon graphs (or, equivalently, cellularly embedded graphs). The theory of minors of ribbon graphs differs from that of graphs in that contracting loops is necessary and doing this can create additional vertices and components. 
Thus the ribbon graph minor relation is incompatible with the graph minor relation. We discuss excluded minor characterisations of minor closed families of ribbon graphs. Our main result is an excluded minor characterisation of the family of ribbon graphs that represent knot and link diagrams.
\end{abstract}

\maketitle

\section{Introduction and overview} 

One of the  deepest results in graph theory, the Robertson-Seymour Theorem \cite{RS}, is that graphs are  well-quasi-ordered under the graph minor relation. This theorem may be reformulated as stating that every minor-closed family of graphs is characterised by a finite set of excluded minors. However, although we know minor-closed families can be characterised by excluded minors, very few  explicit characterisations are known. (Perhaps the best-known being Wagner's Theorem which characterises planar graphs as those with no $K_5$ or $K_{3,3}$ minors.)

Rather than working with abstract graphs, in this paper we consider minors of cellularly embedded graphs which we realise as ribbon graphs. Ribbon graph minors differ from minors of abstract graphs as it is necessary to allow the contraction of loops (forbidding the contraction of loops results in infinite anti-chains, as in Example~\ref{ex.b}, and so ribbon graphs are not well-quasi-ordered under such a relation). Moreover, contracting a loop $e$ of a ribbon graph $G$  may result in ribbon graph $G/e$ with more vertices or components than the original one. Thus the underlying graphs of two ribbon graph minors need not be graph minors. 
We conjecture the analogue of the Robertson-Seymour Theorem for ribbon graphs:  every ribbon graph minor-closed family of ribbon graphs can be characterised by a finite set of excluded ribbon graph minors.  This conjecture leads to the problem of finding excluded ribbon graph minor characterisations of  ribbon graph minor-closed families of ribbon graphs. While it is fairly straightforward to find excluded ribbon graph minor characterisations of some families of ribbon graphs, such as orientable ribbon graphs, of course this is not always the case. Our main result here is an excluded ribbon graph minor characterisation of the set of ribbon graphs that represent knot and link diagrams.

There is a classical, and well-known way  to represent  link diagrams as signed plane graphs (see for example~\cite{Bo,EMMbook,We}). This construction provides a bridge between knot theory and graph theory,  and has found numerous applications in both of these areas.  In the construction, the over/under crossing structure of the link diagram is encoded by signs $+$ or $-$ on the edges of the plane graph. (The link diagram arises as the medial graph of the plane graph with the crossings determined by the signs). By considering graphs in orientable surfaces of higher genus,  Dasbach,  Futer,  Kalfagianni,  Lin and Stoltzfus in \cite{Da} (see also Turaev~\cite{Tu97}) explained how the crossing structure can be encoded in the topology of an embedded graph,  avoiding the need for signed graphs.
This idea has proved to be  very useful  and  has found many recent applications in knot theory, such as to knot polynomials, Khovanov homology, knot Floer homology, Turaev genus, quasi-alternating links,  the signature of a knot, the  determinant of a knot, and to hyperbolic knot theory (see, for example, \cite{Ab09,CKS07,CP,Ch,CV,Da,Detal2,DL10, FKP08,FKP09,HM,Lo08, Mo2,Tu97,VT10,Wi09}). Not only this, but the insights provided by the construction have led to new developments in graph theory (particularly for graph polynomials) and quantum field theory.

Although every link diagram can be represented by an embedded graph, not every embedded graph represents a link diagram. Given the applications, understanding  the class of embedded graphs that represent link diagrams is an important problem. The main result of this paper is an excluded  minor characterisation of this class which we state in terms of ribbon graphs:
\begin{maintheorem}\label{mt1}
Let $B_{\bar1}$, $B_3$, and $\theta_{t}$ be the ribbon graphs shown in Figure~\ref{f.2}.
Then a ribbon graph represents a link diagram if and only if it contains no ribbon graph minor equivalent to $B_{\bar1}$, $B_3$, or $\theta_{t}$.
\end{maintheorem}
Its worth noting that ribbon graphs of link diagrams are necessarily orientable and that $B_{\bar1}$ appears here only to ensure orientability. The Main Theorem could be restated as a ribbon graph represents a link diagram if and only if it is orientable and contains no ribbon graph minor equivalent to $B_3$, or $\theta_{t}$.

\medskip

I would  like to thank Neal Stoltzfus, Adam Lowrance and Mark Ellingham for helpful and stimulating discussions. I would also like to that the referees for their thoughtful comments.

\begin{figure}
\centering 
\subfigure[$B_{\bar{1}}$. ]{
 \includegraphics[scale=.5]{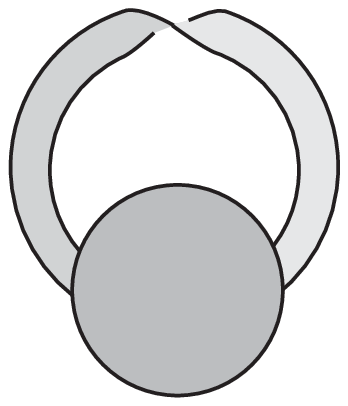}
\label{f.2a}
}
\hspace{1cm}
\subfigure[$B_3$. ]{
\includegraphics[scale=.5]{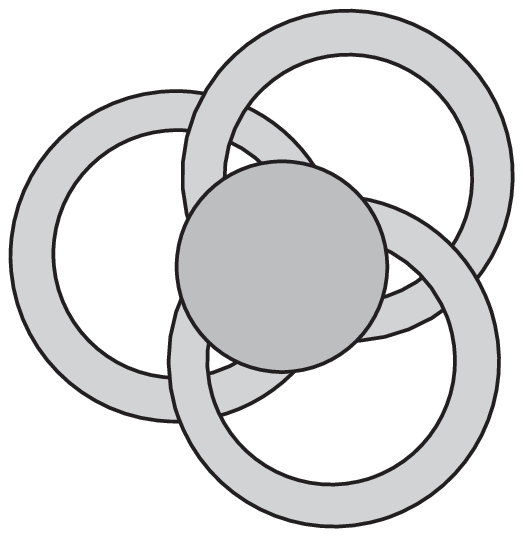}
\label{f.2b}
}
\hspace{1cm}
\subfigure[$\theta_t$. ]{
\includegraphics[scale=.5]{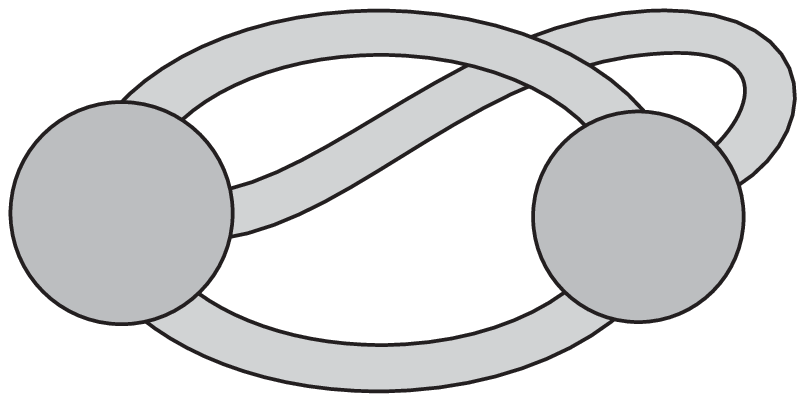}
\label{f.2c}
}
\caption{The excluded ribbon graph minors for the family of ribbon graphs that represent link diagrams.}
\label{f.2}
\end{figure}

\section{Ribbon graphs and their minors}

\subsection{Ribbon graphs}
We assume a familiarity with basic graph theory and topological graph theory, but give a brief review of ribbon graphs referring the reader to  \cite{EMMbook,GT87} for further details. (Note that ribbon graphs are called `reduced band decompositions' in \cite{GT87}.)
A {\em ribbon graph} $G =\left(  V(G),E(G)  \right)$ is a surface with boundary, represented as the union of two  sets of  discs: a set $V (G)$ of {\em vertices} and a set $E (G)$ of {\em edges}  such that: (1) the vertices and edges intersect in disjoint line segments;
(2) each such line segment lies on the boundary of precisely one
vertex and precisely one edge; and (3) every edge contains exactly two such line segments. See Figure~\ref{f.2} for some examples of ribbon graphs.
It is well-known that ribbon graphs are equivalent to cellularly embedded graphs and to band decompositions. (Ribbon graphs and band decompositions arise naturally from neighbourhoods of cellularly embedded graphs. On the other hand, topologically a ribbon graph is a  surface with boundary, capping-off the holes results in a \emph{band decomposition}, which gives rise to a cellularly embedded graph in the obvious way. Again, see \cite{EMMbook,GT87} for details.)  A {\em bouquet} is a ribbon graph with exactly one vertex. A ribbon graph is {\em orientable} if it is orientable when viewed as a surface, and  is {\em plane} if when viewed as a surface it is a  sphere with holes.  The {\em genus}, $g(G)$, of a ribbon graph $G$  is its genus when viewed as a surface. Its {\em Euler genus}, $\gamma(G)$,  is defined as $2g(G)$ if $G$ is orientable and $g(G)$ if $G$ is non-orientable.
Two ribbon graphs are {\em equivalent} if they describe equivalent cellularly embedded graph, and we consider ribbon graphs up to equivalence. This means ribbon graphs are equivalent if there is a homeomorphism taking one to the other that preserves the vertex-edge structure and the cyclic order of the half-edges at each vertex. The homeomorphism should be orientation preserving when the ribbon graphs are orientable.

At times we find it convenient to describe ribbon graphs using Chmutov's arrow presentations from \cite{Ch}. An \emph{arrow presentation} is a set of closed curves,  each with a collection of disjoint,  labelled arrows  lying on them where each label appears on precisely two arrows. A ribbon graph $G$ can be formed from an arrow presentation by  identifying each closed curve with the boundary of a disc (forming the vertex set of $G$). Then, for each pair of $e$-labelled arrows, take a disc (which will form an edge of $G$), orient its boundary, place two disjoint arrows on its boundary that point in the direction of the orientation, and identify each $e$-labelled arrow on this edge.  Conversely a ribbon graph can be described as an arrow presentation by arbitrarily labelling and orienting the boundary of each edge disc of $G$. Then on each arc where an edge disc intersects a vertex disc, place an arrow on the vertex disc, labelling the arrow with the label of the edge it meets and directing it consistently with the orientation of the edge disc boundary. The boundaries of the vertex set marked with these labelled arrows give the arrow-marked closed curves of an arrow presentation.  See Figure~\ref{f4} for an example.

\subsection{Ribbon graph minors}
Let $G$ be a ribbon graph, $e\in E(G)$, and $v\in V(G)$. Then $G-e$ denotes the ribbon graph obtained  from $G$ by  deleting the edge $e$, and  $G-v$ denotes the ribbon graph obtained  from $G$ by deleting the vertex $v$ and all of its incident edges. A ribbon graph $H$ is a {\em ribbon subgraph} of $G$ if it can be obtained from $G$ by deleting vertices and edges.  
If $u_1$ and $u_2$ are the (not necessarily distinct)  vertices incident to $e$, then $G/e$ denotes the ribbon graph obtained as follows:  consider the boundary component(s) of $e\cup u_1 \cup u_2$ as curves on $G$. For each resulting curve, attach a disc (which will form a vertex of $G/e$) by identifying its boundary component with the curve. Delete $e$, $u_2$ and $u_2$ from the resulting complex, to get the ribbon graph $G/e$. We say $G/e$ is obtained from $G$ by {\em contracting} $e$. See Table~\ref{tablecontractrg} for the local effect of contracting an edge of a ribbon graph. 

If $G$ is viewed as an arrow presentation then $G/e$ is obtained as follows. Suppose $\alpha$ and $\beta$ are the two $e$-labelled arrows. Connect the tip of $\alpha$ to the tail of $\beta$ with a line segment; and  connect the tip of $\beta$ to the tail of $\alpha$ with another line segment. Delete $\alpha$, $\beta$ and the arcs of the curves (or curve) on which they lie. The resulting arrow presentation describes $G/e$. Again see  Table~\ref{tablecontractrg}. 

    Note that contraction of non-loop edges of a ribbon graph is compatible with the  standard contraction of non-loop edges in cellularly embedded graphs, and coincides with the obvious idea of contracting a non-loop edge $e=(u_1,u_2)$ by making the disc $e\cup u_1 \cup u_2$ into a vertex.  However, we emphasise that when $e$ is a loop, $G/e$ may have more components and vertices than $G$. In particular, this means that the underlying graphs of two ribbon graph minors (which we define shortly) need not be graph minors, and this is where the main difference between graph minor theory and ribbon graph minor theory originates.

\begin{table}
\centering
\begin{tabular}{|c||c|c|c||c|}\hline
 &  non-loop & non-orientable loop&orientable loop& arrow presentation\\ \hline
\raisebox{6mm}{$G$} &
\includegraphics[scale=.25]{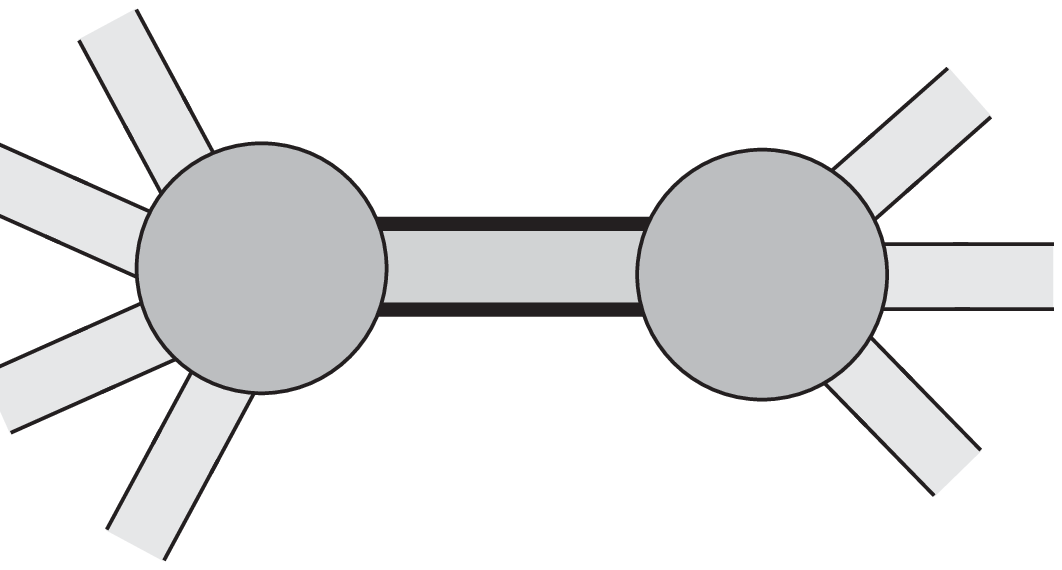} &\includegraphics[scale=.25]{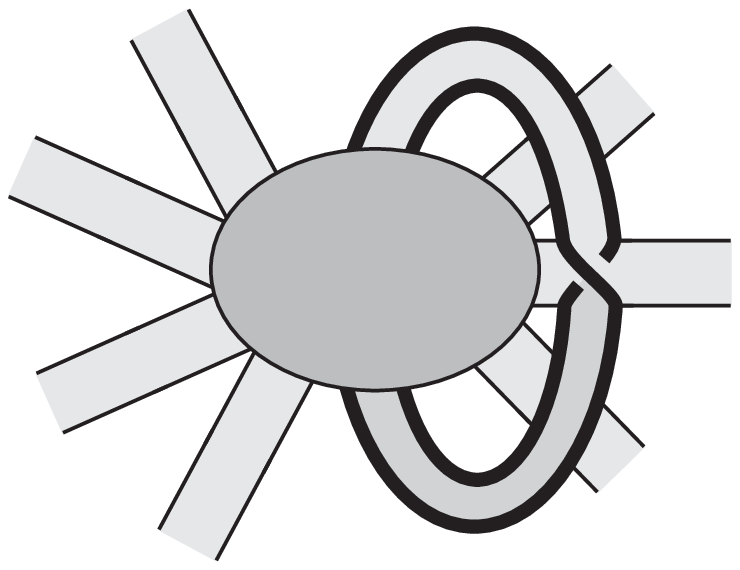}&\includegraphics[scale=.25]{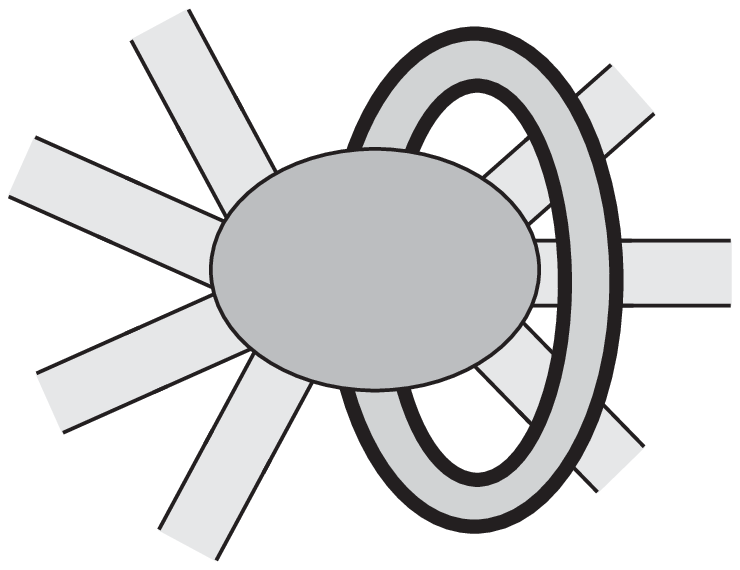} &\includegraphics[scale=.5]{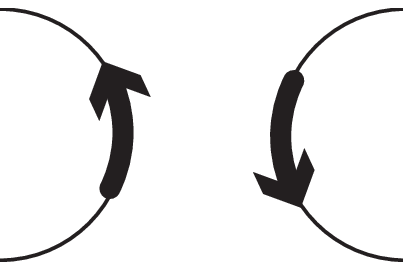}
\\ \hline
\raisebox{6mm}{$G-e$}
&
\includegraphics[scale=.25]{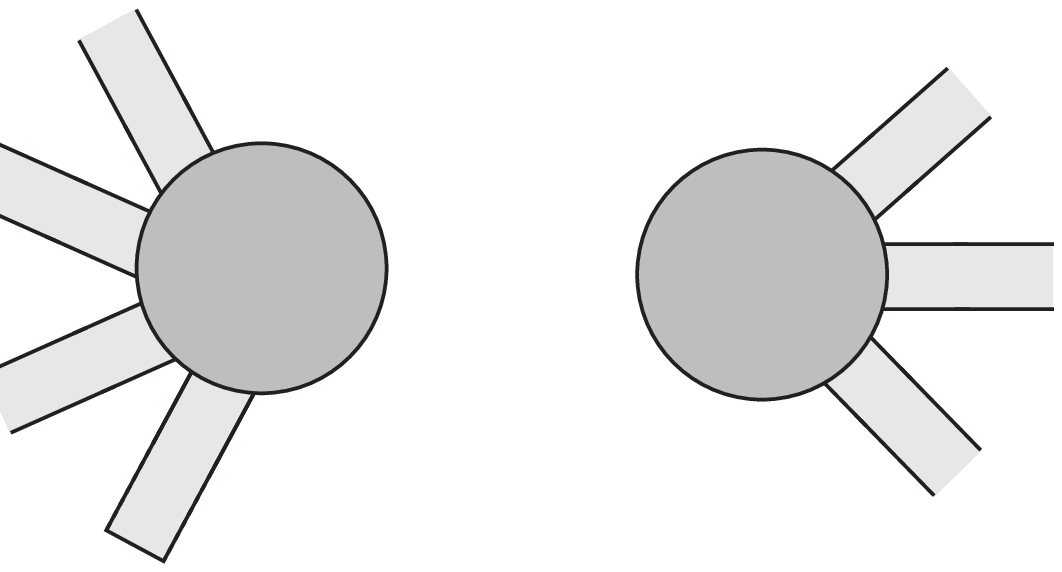} &\includegraphics[scale=.25]{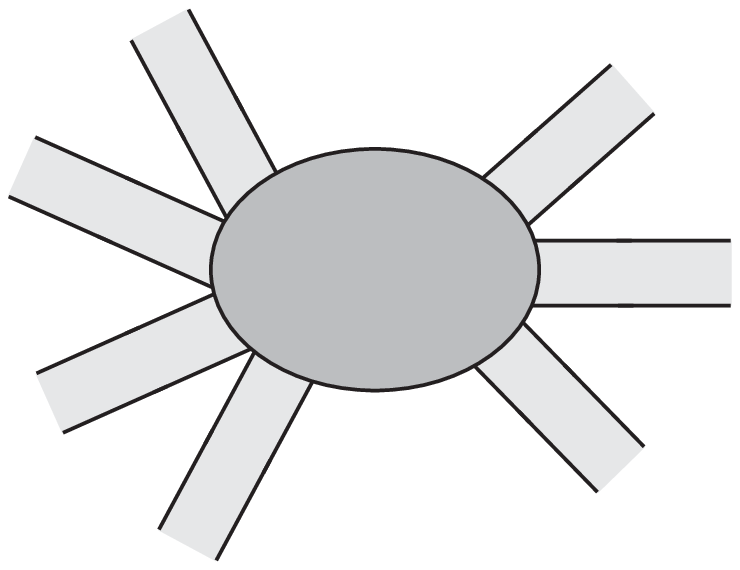}&\includegraphics[scale=.25]{ch4_35a} &\includegraphics[scale=.5]{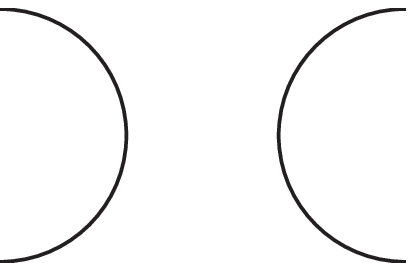}
\\ \hline
\raisebox{6mm}{$G/e$}
&
\includegraphics[scale=.25]{ch4_35a} &\includegraphics[scale=.25]{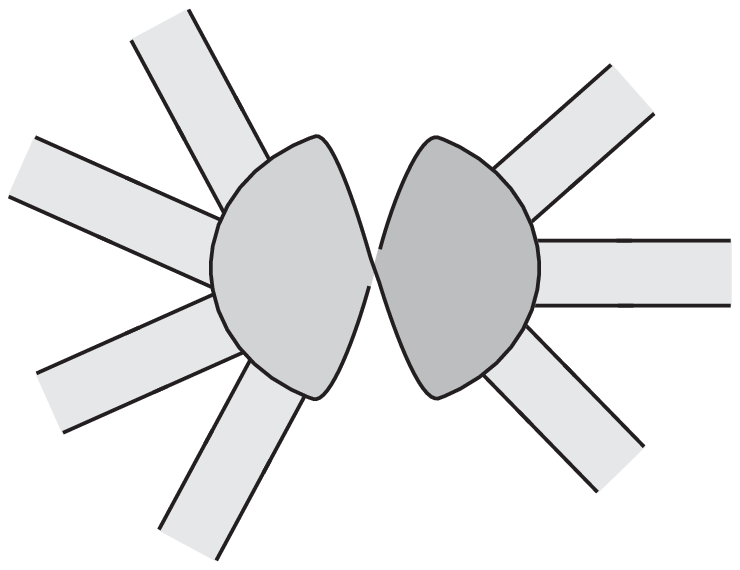}&\includegraphics[scale=.25]{ch4_38a} & \includegraphics[scale=.5]{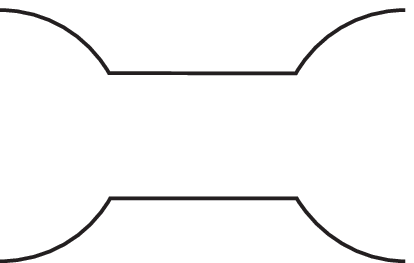}
\\ \hline\hline
\raisebox{6mm}{$G^{e}$} &
\includegraphics[scale=.25]{ch4_35} &\includegraphics[scale=.25]{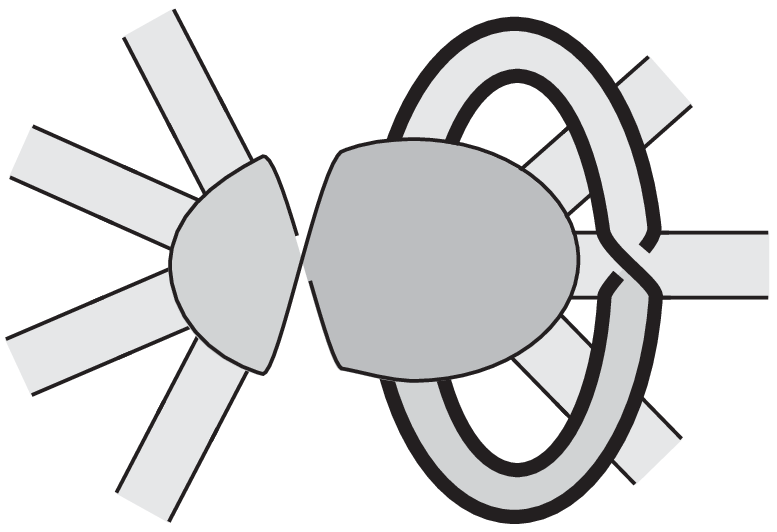} &\includegraphics[scale=.25]{ch4_38} &\includegraphics[scale=.5]{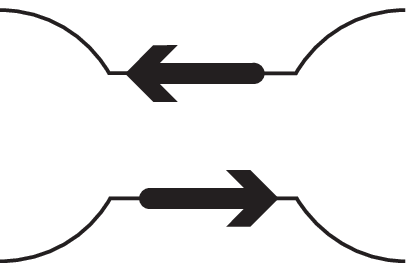}
\\ \hline
\end{tabular}
\caption{Operations on an edge $e$ (highlighted in bold) of a ribbon graph. The ribbon graphs are identical outside of the region shown.}
\label{tablecontractrg}
\end{table}

If multiple edges of a ribbon graph are contracted and/or deleted, the resulting graph does not depend on the order of the deletions and contractions. We say that a ribbon graph $H$ is a {\em ribbon graph minor} of a ribbon graph $G$ if $H$ is obtained from $G$ by a sequence of edge deletions, vertex deletions, or edge contractions. In addition, we say that $G$ has an {\em $H$-ribbon graph minor} if it has a ribbon graph minor equivalent to $H$.
We can assume without loss of generality that in the formation of a ribbon graph minor only isolated vertices are ever deleted. 
A set $S$ of ribbon graphs is {\em ribbon graph minor-closed} if for each $G\in S$ every ribbon graph minor of $G$ is  in $S$.

\medskip

A {\em quasi-ordering} is a reflexive and transitive relation.  A quasi-ordering $\leq$ on a set $X$ is a  {\em well-quasi-ordering} if it contains neither  an infinite antichain nor an infinite decreasing sequence $x_0>x_1>\cdots$. (An {\em  antichain} is a subset with the property that any two elements are incomparable.) The Robertson-Seymour Theorem states that graphs are well-quasi-ordered by the graph minor relation. For graphs, deleting and contracting a loop results in the same graph, and so loops need not be contracted. This is in sharp contrast to ribbon graph minors where  forbidding the contraction of loops results in infinite anti-chains as in the following example.

\begin{example}\label{ex.b}
Let $n\in \mathbb{N}$   and $B_n$  denote the orientable bouquet with edges  $e_1, \ldots, e_n$ that meet the vertex in the cyclic order 
$ e_2e_1e_3e_2e_4e_3\cdots e_{n}e_{n-1}e_1e_n$. (See Figures~\ref{f.2b} and~\ref{f.3a}.) 
Also let $\mathcal{F}=\{ B_{2k+1} \mid k\in \mathbb{N} \}$. If we forbid the contraction of loops in ribbon graph minors,  $\mathcal{F}$ is an  infinite anti-chain. However, when we allow loops to be contracted,  $B_{2k+1}$ is a ribbon graph minor of $B_{2k+3}$, for each $k\in \mathbb{N}$ (see Figure~\ref{f.3}). In particular, it follows  every ribbon graph in $\mathcal{F}$ has a $B_3$-ribbon graph minor. Observe that $B_n/e_n$ is not a graph minor of $B_n$.
\end{example}

\begin{figure} 
\centering 
\subfigure[$B_{n}$. ]{
\labellist \small\hair 2pt
\pinlabel {$e_2$} [l] at    52 3
\pinlabel {$e_1$} [l] at     123 40
\pinlabel {$e_n$} [l] at    154 122
\pinlabel {$e_{n-1}$} [l] at   123 195
\pinlabel {$e_{n-2}$} [l] at   52 240
\endlabellist
 \includegraphics[scale=.4]{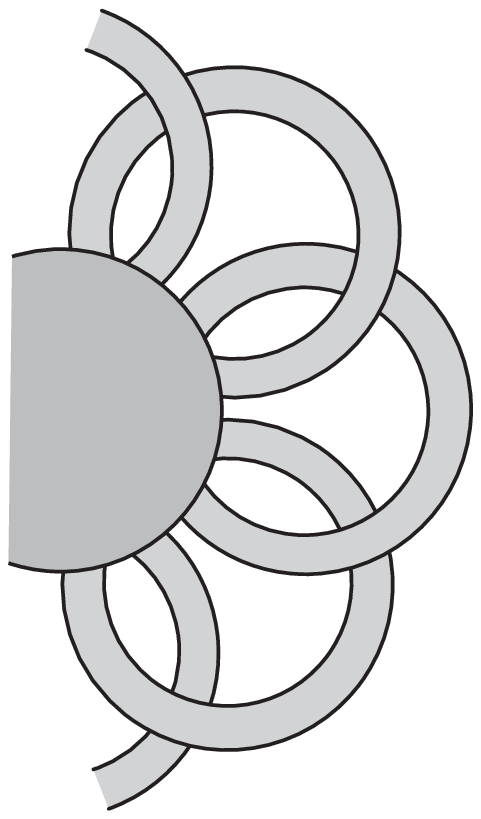}
\label{f.3a}
}
\hspace{1cm}
\labellist \small\hair 2pt
\pinlabel {$e_2$} [l] at    52 3
\pinlabel {$e_1$} [l] at     125 40
\pinlabel {$e_{n-1}$} [l] at   125 195
\pinlabel {$e_{n-2}$} [l] at   52 240
\endlabellist
\subfigure[$B_{n}/e_n$. ]{
\includegraphics[scale=.4]{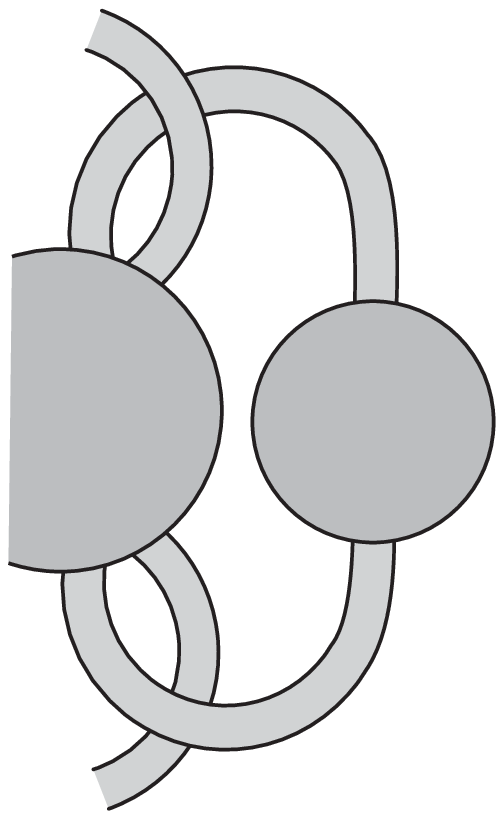}
\label{f.3b}
}
\hspace{1cm}
\labellist \small\hair 2pt
\pinlabel {$e_2$} [l] at     123 40
\pinlabel {$e_1$} [l] at    154 122
\pinlabel {$e_{n-2}$} [l] at   123 195
\endlabellist
\subfigure[$(B_{n}/e_n)/e_{n-1}$  $= B_{n-2}$. ]{
\includegraphics[scale=.4]{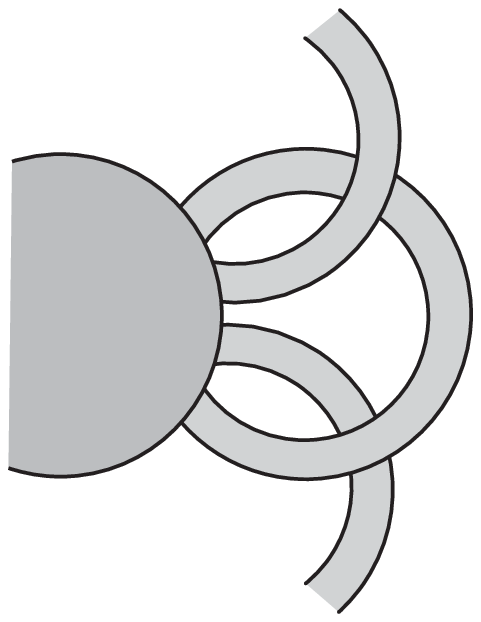}
\label{f.3c}
}
\caption{Recovering $B_{n-2}$ as a  ribbon graph minor of $B_{n}$.}
\label{f.3}
\end{figure}

\begin{conjecture}\label{c}
The set of ribbon graphs is well-quasi-ordered by the ribbon graph minor relation.
\end{conjecture}
One of the consequences of Conjecture~\ref{c}, if it is true, is that every ribbon graph minor-closed class of ribbon graphs can be characterised in terms of a  finite set of excluded ribbon graph minors. The rest of this paper is concerned with finding an excluded ribbon graph minor characterisation of the set of ribbon graphs that represent link diagrams.

\section{An excluded ribbon graph minor characterisation of the ribbon graphs of links}

\subsection{Representing  link diagrams by ribbon graph}

Let $D\subset \mathbb{R}^2$ be a  link diagram. The {\em ribbon graph of $D$} (or the {\em All-A ribbon graph of $D$}), introduced in \cite{Da}, is formed as follows. 
Assign a unique label to each crossing of $D$. An {\em arrow marked A-smoothing} of a crossing $c$ of $D$,   is the replacement of the crossing $c$ with a pair of decorated curves as indicated in  Figures~\ref{f4a} and~\ref{f4b}. The diagrams are identical outside of the region shown. Replacing each crossing of $D$ with its arrow marked A-smoothing  results in an arrow presentation. The corresponding ribbon graph is  $\mathbb{A}(D)$. We say that a ribbon graph $G$ {\em represents a link diagram}, or is the {\em ribbon graph of a link diagram}, if $G=\mathbb{A}(D)$, for some link diagram $D$.
It is easy to see that $\mathbb{A}(D)$ is orientable for each diagram $D$. Note also that the vertices of $\mathbb{A}(D)$ correspond to the closed curves of the A-smoothing of $D$.

\begin{figure}
\centering
\subfigure[A crossing $c$.]{
\labellist
 \small\hair 2pt
\pinlabel {$c$}  at 37 50 
\endlabellist
\quad\quad\includegraphics[scale=0.65]{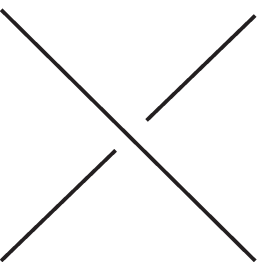} \quad\quad
\label{f4a}
}
\hspace{5mm}
\subfigure[An arrow  marked \mbox{A-smoothing} of $c$. ]{
\labellist
 \small\hair 2pt
\pinlabel {$c$}  at 30 28  
\pinlabel {$c$}  at 38 48 
\endlabellist
\quad\quad\includegraphics[scale=0.65]{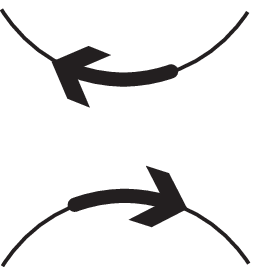} \quad\quad
\label{f4b}
}
\subfigure[A  link diagram $D$.]{
\labellist
 \small\hair 2pt
\pinlabel {$1$} [l] at   78 183
\pinlabel {$2$}   at   15 118
\pinlabel {$3$}  [l] at    107 172
\pinlabel {$4$}  [r] at 32 103
\pinlabel {$5$}  [l] at   147 104
\pinlabel {$6$}  [r] at  58 75
\pinlabel {$8$}  [l] at    117 63
\pinlabel {$7$}  [r] at  62 42
\endlabellist
\includegraphics[scale=.8]{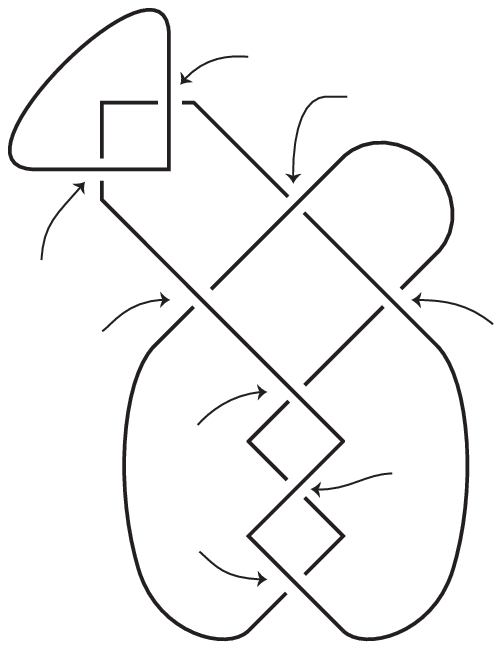}
\label{f4c}
}

\subfigure[An arrow presentation of $\mathbb{A}(D)$. ]{
\labellist
 \small\hair 2pt
\pinlabel {$1$}  at    73 148
\pinlabel {$2$}   at   18 122
\pinlabel {$3$}  at   103 167 
\pinlabel {$4$}  [r] at  35 106
\pinlabel {$5$}   at    149 106
\pinlabel {$6$}  [l] at  117 82
\pinlabel {$7$}  [l] at    117 27
\pinlabel {$8$}  [l] at  118 63
\pinlabel {$8$}   at  62 63
\endlabellist
\includegraphics[scale=.8]{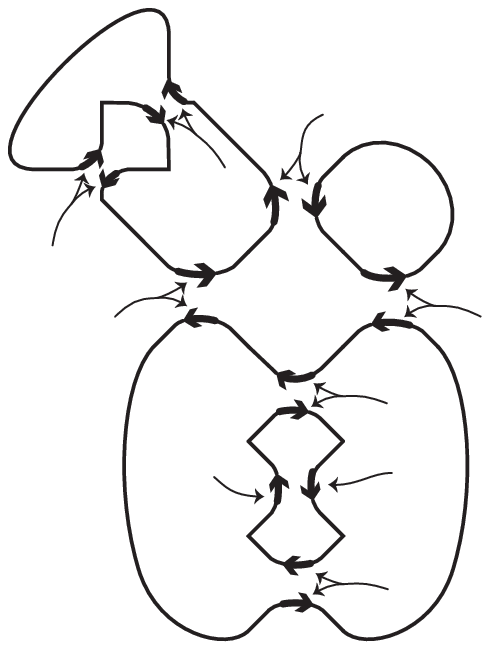}
\label{f4e}
}
\hspace{5mm}
\subfigure[The  ribbon graph $\mathbb{A}(D)$.]{
\labellist
 \small\hair 2pt
\pinlabel {$1$}  at  106 228 
\pinlabel {$2$}   at    25 201
\pinlabel {$3$}  at    144 186
\pinlabel {$4$}   at   83 121
\pinlabel {$5$}   at     146 79
\pinlabel {$6$}  at   131 99
\pinlabel {$7$}   at    166 13
\pinlabel {$8$}   at  250 99
\endlabellist
\includegraphics[scale=.55]{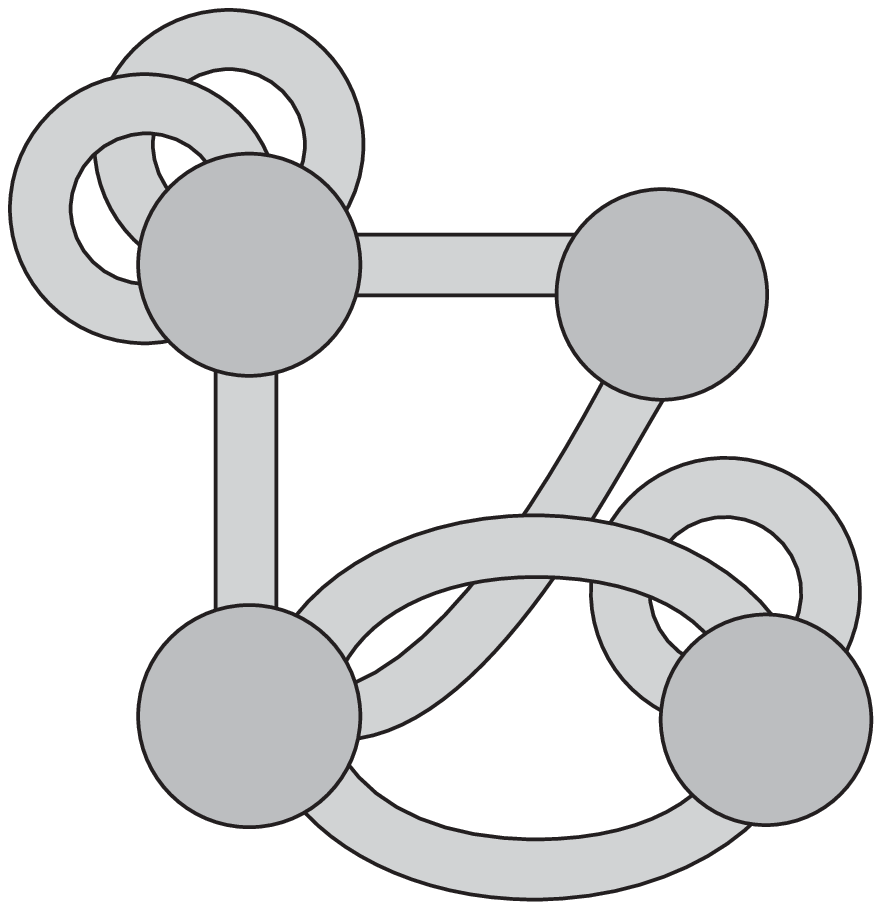}
\label{f4f}
}

\subfigure[An arrow presentation of $\mathbb{A}(D)^{\{ 1,6,7 \}}$. ]{
\labellist
 \small\hair 2pt
\pinlabel {$1$} [l] at    72 186
\pinlabel {$2$}   at   18 122
\pinlabel {$3$}  at   103 167 
\pinlabel {$4$}  [r] at  35 106
\pinlabel {$5$}   at    149 106
\pinlabel {$6$}  [l] at  117 84
\pinlabel {$6$}   at  62 78 %
\pinlabel {$7$}  [l] at    117 37
\pinlabel {$7$}   at    62 37 %
\pinlabel {$8$}  [l] at  118 63
\pinlabel {$8$}   at  62 63
\endlabellist
\includegraphics[scale=.8]{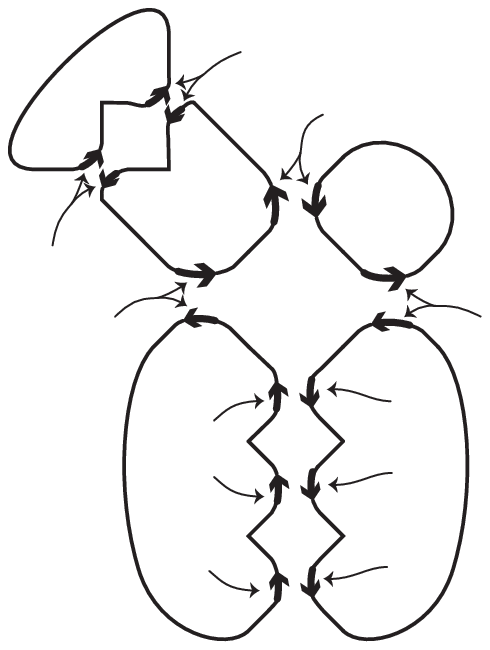}
\label{f4g}
}
\hspace{5mm}
\subfigure[The  ribbon graph $\mathbb{A}(D)^{\{ 1,6,7 \}}$.]{
\labellist
 \small\hair 2pt
\pinlabel {$1$}  at  125 211 
\pinlabel {$2$}   at    53 187
\pinlabel {$3$}  at    160 171
\pinlabel {$4$}   at   109 109
\pinlabel {$5$}   at     215 109
\pinlabel {$6$}  at   161 87
\pinlabel {$7$}   at    161 51
\pinlabel {$8$}   at  161 17
\endlabellist
\includegraphics[scale=.55]{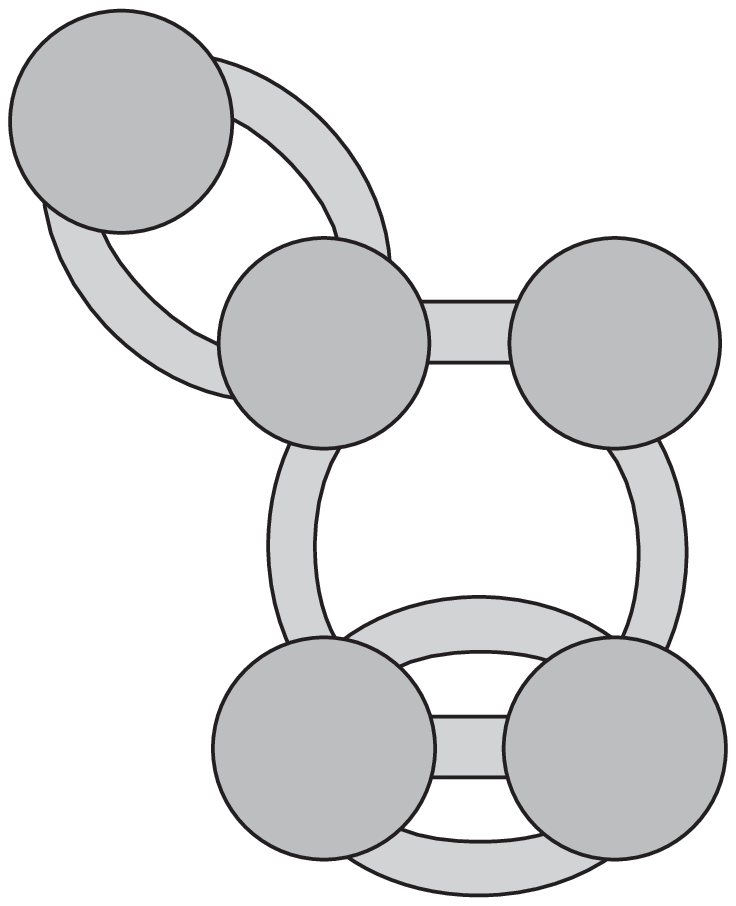}
\label{f4h}
}

\caption{A ribbon graph of a link diagram.}
\label{f4}
\end{figure}

An important observation is that not every ribbon graph arises as the ribbon graph of a link diagram (for example, for $n\geq 3$ there are more ribbon graphs on $n$ edges than there are link diagrams with $n$ crossings). The question of characterising the class of ribbon graphs that represent link diagrams then arises. The following theorem provides an excluded ribbon graph minor characterisation of this class. For the theorem, recall $B_3$ from Example~\ref{ex.b} and Figure~\ref{f.2a}; let $B_{\bar1}$ be the non-orientable bouquet with one edge, as in Figure~\ref{f.2c}; and let $\theta_{t}$ be the toroidal $\theta$-graph, as in Figure~\ref{f.2c}.
\begin{maintheorem}
A ribbon graph represents a link diagram if and only if it contains no ribbon graph minor equivalent to $B_{\bar1}$, $B_3$, or $\theta_{t}$.
\end{maintheorem}
The remainder of this paper is devoted to the proof of this theorem.

\subsection{The connection with partial duals}
When working with the ribbon graphs of link diagrams it is often convenient to use the framework of partial duals. Partial duality was introduced by Chmutov in \cite{Ch}. It arises as a natural operation in knot theory, topological graph theory, graph polynomials, delta-matroids, and quantum field theory.  Roughly speaking, a partial dual of a ribbon graph is obtained by forming the geometric dual with respect to only some of its edges. To make this concrete, let $G=(V(G), E(G))$ be a ribbon graph, $A\subseteq E(G)$ and regard the boundary components of the  ribbon subgraph $(V(G),A)$ of $G$ as curves on the surface of $G$. Glue a disc to $G$ along each of these curves by identifying the boundary of the disc with the curve, and remove the interior of all vertices of $G$. The resulting ribbon graph is the {\em partial dual} $G^{A}$. If $A=\{e\}$ we  write $G^e$ for $G^{\{e\}}$. As an example, if $e$ is any edge of $\theta_t$, then  $\theta_t^e=B_3$.

Partial duality changes the ribbon graph locally at the edges in $A$ and the regions where they meet their incident vertices.  Table~\ref{tablecontractrg} shows the local effect of forming the partial dual with respect to an edge $e$ of a ribbon graph $G$.

Observe from the table (or from the definitions) that $G/e = G^e-e$. We also have that the edges of $G^A$ correspond to the edges of $G$; 
 $G^*=G^{E(G)}$, where $G^*$ is the geometric dual of $G$; $G^{\emptyset}=G$; $ (G^A)^B=  G^{A\triangle B}$ (in particular partial duals can be formed one edge at a time); and partial duality acts disjointly on connected components. 

If $G$ is viewed as an arrow presentation then $G^A$ is obtained as follows. For each $e\in A$, suppose $\alpha$ and $\beta$ are the two $e$-labelled arrows. Place an $e$-labelled arrow from  tip of $\alpha$ to the tail of $\beta$, and an $e$-labelled arrow from  tip of $\beta$ to the tail of $\alpha$.   Delete $\alpha$, $\beta$ and the arcs of the curves (or curve). The resulting arrow presentation describes $G^A$. See  Table~\ref{tablecontractrg} and Figure~\ref{f4}.

Our interest in partial duals here arises from the following proposition. It was formally proven in \cite{Mo12}, although all the ideas behind the result were present in  \cite{Ch}, and for the convenience of the reader we provide a proof here.
\begin{proposition}\label{p1}
A ribbon graph $G$ represents a link diagram if and only if it is a partial dual of a plane ribbon graph.
\end{proposition}
\begin{proof}
 Defining  arrow marked B-smoothings analogously to arrow marked A-smoothings (just switch the type of smoothing in Figure~\ref{f4b}),  first observe that checkerboard colouring (i.e., face 2-colouring) $D$ and choosing arrow marked smoothings  that follow the black faces of $D$ at each crossing results in an arrow presentation of a plane ribbon graph. Next observe (recalling Table~\ref{tablecontractrg}) that forming the partial dual in the language of arrow presentations with respect to the edges of the  B-smoothed crossings results in an arrow presentation of $\mathbb{A}(D)$. The result follows. (See Figure~\ref{f4} for an illustration of this argument.)
\end{proof}

We now give a few lemmas about partial duals. The first of which says that 
the ribbon graph minors of a partial dual of $G$ are the partial duals of the ribbon graph minors of $G$. For the lemma we introduce the notation that if $H$ is a ribbon graph minor of $G$, and $A\subseteq E(G)$, then by $H^A$ we mean $H^{A\cap E(H)}$.
\begin{lemma}\label{l1}
Let $G$ be a ribbon graph and $A\subseteq E(G)$. Then 
\[ \{ J^A \mid J \text{ is a ribbon graph minor of } G \}  =  \{ H \mid H \text{ is a ribbon graph minor of } G^A \}   \]
\end{lemma}
\begin{proof}
The result is easily verified when $|E(G)|\in \{0,1\}$ or  $A=\emptyset$, so assume that this is not the case.
Since partial duals can be formed one edge at a time it is enough to prove the theorem for $A=\{e\}$. Let $f\in E(G)$ with $f\neq e$. Since edge deletion, contraction and partial duality change the ribbon graph locally at the edge involved, we have that $(G-f)^e=G^e-f$ and $(G/f)^e=(G^e)/f$. 
Next 
$ (G-e)^e = G-e = (G^{e})^e-e = (G^e)/e$, 
where the first equality is by definition as $e\not\in E(G-e)$, the second since partial duality is involutary, and the third by the relation  $G/e = G^e-e$ between partial duality and contraction. Similarly, $(G/e)^e=G/e=(G^e)-e$.
From these identities it follows that  $H$ is obtained from $G^A$ by deleting and contracting edges if and only if $H=J^A$ from some $J$ obtained from $G$ by deleting and contracting edges.  Since deleting isolated vertices of $H$ and $J^A$ correspond, this statement also holds when the operation of vertex deletion is included, and the result follows.
\end{proof}

\begin{lemma}\label{l3}
Let $G$ be a ribbon graph. 
\begin{enumerate}
\item If $H$ is a ribbon graph minor of $G$, then $g(H)\leq g(G)$  (respectively, $\gamma(H)\leq \gamma(G)$). In particular, for each $k\in \mathbb{N}_0$ the set of ribbon graphs of genus (respectively, Euler genus) at most $k$ is ribbon graph minor-closed.
\item \label{l3b} For each $k\in \mathbb{N}_0$ the set 
of all ribbon graphs that have a partial dual of genus (or Euler genus) at most $k$ is ribbon graph minor-closed.
\end{enumerate}
\end{lemma}
\begin{proof}
For the first item of the lemma begin by noting that deleting an isolated vertex does not change the genus of a ribbon graph. It then remains to show that deleting or contracting an edge can not increase genus, but this follows since $G/e$ and $G-e$ naturally embed in $G$ (examine Table~\ref{tablecontractrg}).  Note that this item can also be proven using Euler's formula.

For the second item of the lemma, let $G$ be a ribbon graph. The result is trivial if $E(G)=\emptyset$, so assume that this is not the case. Let $e\in E(G)$ and  suppose that $g(G^A)\leq k$ for some $A\subseteq E(G)$. 
If $e\notin A$, then since partial duality and edge contraction act locally $(G/e)^A=G^A/e$,
so by the first item of the lemma, $g((G/e)^A) = g(G^A/e)\leq  g(G^A)=k$. 
Similarly $(G-e)^A = G^A-e$ giving that $g((G-e)^A)\leq  g(G^A)=k$.
If $e\in A$, then 
$(G/e)^{A\backslash \{e\} }=(G^e-e)^{A\backslash \{e\} }=G^{A}-e$,
giving that $g((G/e)^{A\backslash \{e\} })\leq g(G^A)=k$.
Similarly, 
$(G-e)^{A\backslash \{e\} }=G^{A\backslash \{e\} }-e=(G^{A})^e-e=G^A/e $, giving that $g((G-e)^{A\backslash \{e\} })\leq g(G^A)=k$.
In all cases we see that $G/e$ and $G-e$ have a partial dual of genus at most $k$. Finally, if $v$ is an isolated vertex then $g((G-v)^A) = g((G^A-v)) = g(G^A) = k$. It follows from this that if $G$ has a partial dual of genus $k$, then each ribbon graph minor of $G$ has a partial dual of genus at most $k$. The Euler genus claim can be obtained by replacing genus with Euler genus in this argument.
\end{proof}

\begin{lemma}\label{l4}
The set of ribbon graphs that represent link diagrams is ribbon graph minor-closed.
\end{lemma}
\begin{proof}
The result follows immediately from Propostion~\ref{p1} and Lemma~\ref{l3}\eqref{l3b} with $k=0$.
\end{proof}

\subsection{The proof of the main theorem}

To prove the main theorem, we use a rough structure theorem for the partial duals of plane ribbon graphs from \cite{Mo12} (see also \cite{Mo13}). This rough structure theorem guarantees a decomposition of a partial dual of a plane ribbon graph  into a set of plane ribbon graphs. To describe the result we need  some additional notation.  
A vertex $v$ of a ribbon graph $G$ is  a {\em separating vertex} if there are non-trivial ribbon subgraphs $G_1$ and $G_2$ of $G$ such that $G=G_1\cup G_2$ and $G_1\cap G_2=\{v\}$. 
Let $G$ be a ribbon graph and $A\subseteq E(G)$. 
We use $G|_A$ to denote the restriction of $G$ to $A$, i.e.,  the ribbon subgraph of $G$ that consists exactly of the edges in $A$ and their incident vertices.  
We say that  $A$ {\em defines a plane-biseparation} of $G$ if 
(1) all of the components of $G|_A$ and $G|_{A^c}$ are plane,
(2) every vertex of $G$ that is in both $G|_A$ and $G|_{A^c}$ is a separating vertex of $G$.
We say that a ribbon graph $G$ {\em admits a plane-biseparation} if  there is some $A\subseteq E(G)$ that defines a plane-biseparation of $G$.
As an example,  the ribbon graph in Figure~\ref{f4f} admits plane-biseparations. The edge sets $\{1,6,7\}$, $\{2,6,7\}$, $\{2,3,4,5,8\}$, and $\{1,3,4,5,8\}$ are exactly those that define plane-biseparations.

\begin{theorem}[\cite{Mo12} Theorem~6.1]\label{t2}
Let $G$ be a  ribbon graph and $A\subseteq E(G)$. Then $G^A$ is a plane ribbon graph if and only if 
 $A$ defines a plane-biseparation of $G$.
\end{theorem}

Since Theorem \ref{t2} is a key result in the proof of Theorem~\ref{mt1} we  pause briefly to describe some of the intuition behind it. 
 Recall that $G^A$ is formed by gluing discs to $G$ to cap of each hole of $(V(G),A )=G-A^c$, and then removing the interiors of all vertices of $G$.   If $G^A$ is plane its two ribbon subgraphs $G^A-A^c$ and $G^A-A$ are plane. Since $G^A-A^c$ equals the geometric dual of $G-A^c$, we have that $G-A^c$, and so $G|_A$, must be plane. The planarity of the ribbon subgraph $G^A-A$, consisting of the non-dualled edges,   can be ensured by insisting that  $G-A$, and so $G|_{A^c}$, is plane (see  Figure~9 of \cite{Mo12} for an indication why).  Finally it must be that no homology arises from the fact that $G^A$ is union of  the two ribbon subgraphs $G^A-A^c$ and $G^A-A$. This can be ensured by insisting that  $G|_A$ and $G|_{A^c}$ only meet at separating vertices of $G$ (see  Figure~7 of \cite{Mo13} for an indication why). Combining these three conditions gives the definition of a plane-biseparation, and a formal proof of Theorem~\ref{t2} can be based upon these ideas.

A further concept we need for the proof of the main theorem is that of an intersection graph. Let $G$ be a bouquet. The {\em intersection graph} $I(G)$ of $G$  is the graph with vertex set $E(G)$ and in which two vertices $e$ and $f$ of $I(G)$ are adjacent if and only if their ends are met in the cyclic order $e\, f\, e\; f$ when travelling round the boundary of the unique vertex of $G$. (If $G$ is viewed as a chord diagram, then $I(G)$ is exactly the intersection graph of this chord diagram.)

We can now prove the main theorem.
\begin{proof}[Proof of  Theorem~\ref{mt1}]
First, since the set of ribbon graphs that represent link diagrams is ribbon graph minor-closed (by Lemma~\ref{l4}) and $B_{\bar1}$, $B_3$, or $\theta_{t}$  do not present link diagrams (since, by direct computation or by observing they do not admit plane-biseparations, they are not partial duals of plane ribbon graphs), it follows that if $G$ represents a link diagram it can contain no ribbon graph minor equivalent to $B_{\bar1}$, $B_3$, or $\theta_{t}$.

For the converse, suppose that $G$ does not represent a link diagram. Without loss of generality we assume that $G$ is connected. If $G$ is non-orientable then it contains a cycle $C$ homeomorphic to a M\"obius band. Deleting all of the vertices and edges of $G$ not in $C$ then contracting all but one edge of $C$ gives a $B_{\bar1}$-ribbon graph minor of $G$. Now suppose that $G$ is orientable. 
By Proposition~\ref{p1}, $G$ is not a partial dual of a plane ribbon graph.  Let $T\subseteq E(G)$ be the edge set of a spanning tree of $G$. Then $G^T$ has exactly one vertex. Moreover, $G^T$ is not a partial dual of a plane ribbon graph since $G$ is not, and so by Theorem~\ref{t2} we have that $G^T$ does not admit a plane-biseparation. It follows that its intersection graph $I(G^T)$ is not bipartite. (If $I(G^T)$ was bipartite then consider a proper 2-colouring of it. The set of vertices of each colour each induce a subgraph of $G^T$.  Since no vertices of a given colour in $I(G^T)$ are adjacent each of the induced subgraphs of $G^T$ is plane. It follows that the edge set of $G^T$ determined by the vertices of a given colour induces a plane-biseparation of $G^T$, a contradiction, c.f. the proof of Proposition~6 of \cite{deF}.) 
Since $I(G^T)$ is not bipartite it contains an odd cycle  of length at least 3. Let $C$ be a minimal odd cycle. The vertices in $C$ induce a subgraph of $G^T$ and since $C$ is a minimal odd cycle this subgraph is equivalent to $B_{2k+1}$, for some $k\in \mathbb{N}$.  Hence, by Example~\ref{ex.b}, $G^T$ has a $B_3$-ribbon graph minor. 
Finally, by Lemma~\ref{l1}, it follows that $G$ has a ribbon graph minor that is equivalent to a partial dual of $B_3$, but up to equivalence $B_3$ has exactly two partial duals, itself and  $\theta_{t}$. Thus we have shown that  $G$ contains a ribbon graph minor equivalent to $B_{\bar1}$, $B_3$, or $\theta_{t}$ completing the proof.
\end{proof}

\subsection{Excluded minors for genus}

Lemma~\ref{l3} gave that ribbon graphs of bounded genus or Euler genus are minor closed. For completeness, we discuss excluded minor characterisations for these classes. Recall that both genus and Euler genus are additive over connected components, and  Euler's formula 
$v(G)-e(G)+f(G)=2k(G)-\gamma(G)$ where   $v(G)$, $e(G)$, $f(G)$, and $k(G)$ are the numbers of vertices, edges, boundary components, and connected components of $G$, respectively.

Let $\mathcal{B}_{n}$ denote the set of ribbon graphs  in which each connected component is a non-trivial bouquet with exactly 1 boundary component; and such that  every element is of Euler genus $n+1$ if $n$ is odd, and if $n$ is even the orientable elements are of Euler genus $n+2$, and non-orientable elements are of Euler genus $n+1$.
In addition, let $\mathcal{B}^{o}_{n}$ denote the set of orientable elements of $\mathcal{B}_{n}$.

\begin{theorem}\label{t.g} The following hold.
\begin{enumerate}
\item\label{t.g1} A ribbon graph is orientable if and only if it contains no ribbon graph minor equivalent to $B_{\bar1}$.

\item \label{t.g2}  A ribbon graph has  Euler genus at most $n$ if and only if it contains no ribbon graph minor equivalent to an element of $\mathcal{B}_{n}$.

\item \label{t.g3}  An orientable ribbon graph has  genus at most $n$ if and only if it contains no ribbon graph minor equivalent to an element of $\mathcal{B}^{ o}_{2n}$.
\end{enumerate}
\end{theorem}
\begin{proof}
For each of the three items, the given class of ribbon graphs is ribbon graph minor-closed. The given excluded ribbon graph minors do not belong to the class and so cannot appear as ribbon graph minors. It remains to show that for each  item a ribbon graph not belonging to the given class has a ribbon graph minor belonging to the given set of excluded ribbon graph minors. The argument for Item~\ref{t.g1} appears in the third sentence of the Proof of  Theorem~\ref{mt1}.

For Item~\ref{t.g2},  suppose we are given a ribbon graph $G$ with $\gamma(G)>n$. Choose a spanning tree in each connected component and contract $G$ along all of the edges in the spanning trees. This results in a ribbon graph $G'$ in which each connected component is a bouquet and such that $\gamma(G)=\gamma(G')$. 

\noindent\underline{Claim:} Suppose $B$ is a  bouquet with $\gamma(B)>0$. If $B$ is orientable  then it has a ribbon graph minor of Euler genus $\gamma(B)-2$, and  if $B$ is a non-orientable then it a ribbon graph minor of Euler genus $\gamma(B')-1$. 

Assume the claim is true for the moment.  By its repeated application we can find a minor $G''$ of $G'$ such that $\gamma(G'')$ equals $n+1$ if $n$ was odd;  and if $n$ was even, either $\gamma(G'')=n+2$ with $G''$  orientable, or   $\gamma(G'')=n+1$ with $G''$  non-orientable. Then by repeatedly deleting any edges that intersect two boundary components (which, by Euler's Formula, does not change the Euler genus of the ribbon graph) and deleting any isolated vertices we obtain the required ribbon graph minor in $\mathcal{B}_{n}$. Item~\ref{t.g3}  follows by disregarding  non-orientable  ribbon graphs. 

It remains to prove the claim. Given $B$, recursively delete any edges that intersect two boundary components to obtain a bouquet $B'$ of the same Euler genus but with exactly one boundary component. If $B'$ is orientable then deleting any edge of it will split the boundary component into two, dropping the Euler genus by 2. If $B'$ is non-orientable then deleting an edge might increase the  number of boundary components, or it might not change the number of boundary components. We need to show that there is always some edge whose deletion does not change this number, and so its deletion drops the Euler genus by 1. For this consider the \emph{handle slide} as defined in Figures~\ref{f5a} and~\ref{f5b} which `slides' the end of one edge over an adjacent edge (we make no assumption on the order that the points $1, \ldots , 6$ appear on the vertex of $B'$). Handle slides do not change the number of boundary components or the Euler genus. By checking all of the ways in which the points  $1, \ldots , 6$ can be connected to each other in the unique boundary component of $B'$, observe that if there is an edge whose removal does not change the number of connected components before the handle slide then there is  an edge whose removal does not change the number of boundary components after the handle slide. (For example, if the boundary components is given by connecting 1 to 3, 2 to 5 and 4 to 6, then deleting either edge in Figure~\ref{f5a} and the edge $b$, but not $a$, in Figure~\ref{f5b} preserves the number of boundary components. Note that the statement is false if $B'$ has more than one boundary component, for example,  if the boundary components are  given by connecting 1 to 3, 2 to 6 and 4 to 5,  then deleting the edge $a$ in Figure~\ref{f5a} preserves the number of boundary components, but deleting either edge of in Figure~\ref{f5b}  will change the number of boundary components.) Thus to prove the claim it is enough to show that handle slides can be used to convert $B'$ in to a bouquet with one boundary component that contains a non-orientable loop whose ends are adjacent (deleting this loop clearly does not change the number of boundary components). A standard exercise in low-dimensional topology is showing that handle slides can put any bouquet (often called a disc-band surface in the context of topology) in the normal form illustrated in Figure~\ref{f5c}. The claim and theorem then follow. 
\end{proof}

\begin{figure}
\centering
\subfigure[Slide $b$ over $a$ to the left.]{
\labellist
 \small\hair 2pt
\pinlabel {$1$}  [r] at  10 17 
\pinlabel {$2$} [l]  at   55 17 
\pinlabel {$3$} [r] at    81 17 
\pinlabel {$4$}  [l] at   126 17 
\pinlabel {$5$}  [r] at    154  17 
\pinlabel {$6$}  [l] at   199 17 
\pinlabel {$a$}   at  68 41
\pinlabel {$b$}   at  144 41 

\endlabellist
\includegraphics[scale=0.8]{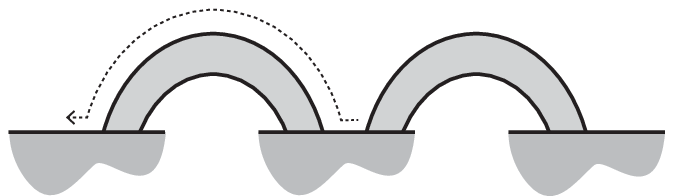}
\label{f5a}
}
\hspace{5mm}
\subfigure[Slide $b$ over $a$ to the right.]{
\labellist
 \small\hair 2pt
\pinlabel {$1$}  [r] at  10 17 
\pinlabel {$2$} [l]  at   55 17 
\pinlabel {$3$} [r] at    81 17 
\pinlabel {$4$}  [l] at   126 17 
\pinlabel {$5$}  [r] at    154  17 
\pinlabel {$6$}  [l] at   199 17 
\pinlabel {$a$}   at  68 41
\pinlabel {$b$}   at  98 73 
\endlabellist
\includegraphics[scale=.8]{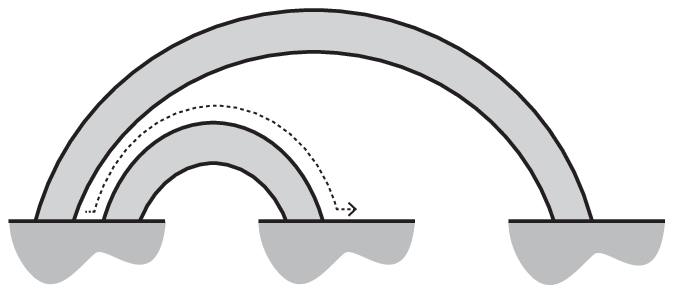}
\label{f5b}
}

\subfigure[A normal form for a ribbon graph.]{
\includegraphics[scale=1]{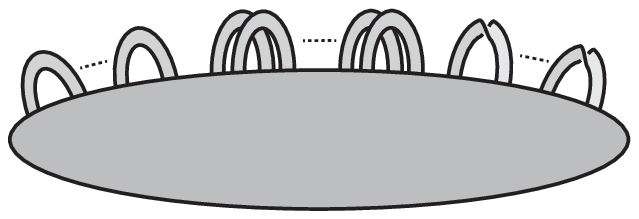}
\label{f5c}
}

\caption{Handle slides.}
\label{f5}
\end{figure}

\end{document}